\documentclass[10pt,reqno]{amsart}

\usepackage{amsmath,amsfonts,latexsym,amssymb}

\usepackage[utf8]{inputenc} 
\usepackage[T1]{fontenc}

\usepackage{dsfont} 
\usepackage{pxfonts}
\usepackage{microtype}
\usepackage{ae,aecompl}

\newtheorem{theorem}{Theorem}[subsection]
\newtheorem{lemma}[theorem]{Lemma}
\newtheorem{proposition}[theorem]{Proposition}

\newcommand{\T}{\ms T}

\renewcommand{\epsilon}{\varepsilon}

\newcommand{\ms}{\mathsf}

\newcommand{\id}{\operatorname{Id}}

\newcommand{\defeq}{\coloneqq}

\newcommand{\PSL}{{\mathsf{PSL}_2(\mathbb R)}}
\newcommand{\CS}[1]{\operatorname{CS}_M\big(#1\big)}
\newcommand{\CSN}[1]{\operatorname{CS}_N\big(#1\big)}

\title[On Tholozan's volume formula]{On Tholozan's volume formula\\ for \\closed anti-de-Sitter 3-manifolds} 
\author{Fran\c cois LABOURIE}
\thanks{}
\begin{document}
\maketitle
\vskip 1 truecm
\vskip 0.2 truecm
This short note is meant as an appendix to Nicolas Tholozan's article \cite{Tholozan:2023} in the same volume. The point here  is to present arguments relating volumes of anti-de-Sitter manifolds of dimension 3 to Chern-Simons invariants, and to  give as a corollary Tholozan's volume formula from his thesis \cite{Tholozan:2014aa}, later extended to higher dimensions in \cite{Tholozan:2018aa}. After hearing his defence, I presented this proof in a seminar in MSRI in 2015 and was planning to save the writing of it for Nicolas's 60th birthday conference, but finally decided against.  

Following here the convention of  \cite{Kulkarni:1985aa, Mess:2007kx,Kassel:2009aa,Salein:2000vm}, an {\em anti-de-Sitter 3-manifold}  (In short AdS manifold) is a manifold of dimension 3, modelled on $\PSL$ equipped with its Killing metric normalised so that the projection from $\PSL$ to $\mathbb H^2$, equipped with its hyperbolic metric, is a metric  submersion.  Hence, $\PSL$ is a Lorentz 3-manifold whose group of isometries is, up to finite covers and quotients,  $\PSL\times\PSL$ where each factor corresponds respectively to the action on the right and on the left on the group $\PSL$. From the homogeneity of the action on timelike vectors, it follows that $\PSL$, hence any AdS 3-manifold, has constant curvature.

Recall that  by a theorem of Bruno Klingler  \cite{Klingler:1996aa} every closed AdS 3-manifold is complete, and thus a quotient of the universal cover $M_0$ of $\PSL$ by a discrete subgroup of $M_0\times M_0$.

Ravi Kulkarni and Frank Raymond \cite{Kulkarni:1985aa}, followed by Franois Salein \cite{Salein:1999aa}, have obtained further restrictions on the possible discrete subgroups appearing, and large classes of examples were constructed in \cite{Salein:2000vm}. Later Fanny Kassel \cite{Kassel:2009aa} obtained the full classification: namely, any closed AdS 3-manifold $M$ is, up to finite covering, a quotient of $\PSL$ by a discrete subgroup of $\PSL\times\PSL$ of the form $(\rho,\sigma)(\pi_1(S))$ where $S$ is a closed surface of negative Euler characteristic, $\rho : \pi_1(S) \to  \PSL$ is a Fuchsian representation, and $\sigma : \pi_1(S) \to  \PSL$ is a non-Fuchsian representation, such that there exists a $(\rho,\sigma)$-equivariant map from ${\bf H}^2$ to ${\bf H}^2$ with Lipschitz constant less than 1. This easily implies that $M=M(S,k)$ is a circle bundle  over $S$, of Euler number $k \neq 0$ (see \cite{Gueritaud:2017aa}). 
\vskip 0.2 truecm

Nicolas Tholozan \cite{Tholozan:2014aa}, answering a question in \cite{Andersson:2007tt}, gave the following striking formula for the volume of the circle bundle  $M(S,k)$ of Euler characteristic $k$ over a surface $S$ of Euler characteristic $e$,  with an AdS structure associated to representations  $\rho$ and $\sigma$ of $\pi_1(S)$ in $\PSL$ of respective Euler class $e$ and $f$:
\begin{eqnarray}
	\operatorname{Vol}\left(M(S,k)\right)= \frac{4\pi^2}{k}\left(e^2-f^2\right)\ . \label{eq:Tholoz}
\end{eqnarray}
As a corollary of his formula, Tholozan obtains a rigidity result: {\em the volume is constant under continuous deformation of any AdS  3-manifold}. Tholozan later on  generalised this formula and rigidity to higher dimensions.

We will show how this low dimensional case of Tholozan's general result follows from considerations on the Chern--Simons invariant. We will omit important technical details. 

I thank Bill Goldman, Fanny Kassel,  Nicolas Tholozan and Jérémy Toulisse for their help and interest.

\section{Chern--Simons invariant in dimension 3}
We present a short and condensed version of the theory and refer to the original article by Shiing Shen Chern and James Simons \cite{Chern:1974aa} for a more careful description or \cite{Labourie:2013ka} for an elementary one in dimension 3.

Let $M$ be a closed oriented 3-manifold equipped with a  vector bundle  $E$. We denote by $\mathcal D$ the affine space of connections on $E$ and define the tangent space to $\mathcal D$ at a connection $\nabla$ as $\mathsf T\mathcal D\defeq \Omega^1(M,\operatorname{End}(E))$.
\subsection{Chern--Simons form}
Given a connection $\nabla$, we now consider the linear map  $\omega^{CS}$ from $\mathsf T\mathcal D$ to $\mathbb R$, given by 
$$
 A\mapsto \int_M \operatorname{trace}\ (A\wedge R^\nabla)\ ,
$$
where $R^\nabla$ is the curvature of $\nabla$ seen as an element of $\Omega^2(M,\operatorname{End}(E))$, and $A\wedge R^\nabla$ the 3-form on $M$  with value in $\operatorname{End}(E)$ given by 
$$
A\wedge R^\nabla(X_1,X_2,X_3)=\frac{1}{6}\sum_{\sigma\in\mathfrak S_3} (-1)^{\epsilon(\sigma)}A(X_{\sigma(1)})R^\nabla(X_{\sigma(2)},X_{\sigma(3)})\ .
$$
The starting result of Chern--Simons theory is
\begin{proposition}
 	Given two gauge equivalent connections $\nabla_0$ and $\nabla_1$ and a path   $\gamma$ of connections joining $\nabla_0 $ to $\nabla_1$ then
\begin{eqnarray}
	\frac{1}{8\pi^2}\int_\gamma \omega^{CS}\in\mathbb Z\ .\label{eq:pontry}
\end{eqnarray}
\end{proposition}
\begin{proof} We only sketch the proof: we interpret a path $\{\nabla_t\}_{t\in[0,1]}$ between two gauge equivalent connections  as a connection $\nabla$ on the  bundle $E$ on $W\defeq M\times S^1$ \cite[proposition 7.2.6.]{Labourie:2013ka}. More precisely 
 $\nabla$ is defined by the following procedure: for $(m,s)$ in $M\times S^1$, we write   $\T_{(m,s)}W=\T M\oplus \mathbb R$, then we identify the space of sections of $E$ on $W$, as $C^\infty(S^1,\Gamma(M,E))$, finally we  define for a section $\sigma$ of $E$
$$
\nabla_{(U,1)} \sigma\defeq\nabla^t_U\sigma(s)+\left(\frac{\partial \sigma}{\partial t}\right)(s)\ .
$$
From this interpretation, we get
\begin{eqnarray}
\int_\gamma \omega^{CS}=\int_{M\times S^1}\operatorname{trace}\ \left(R^{\nabla}\wedge R^{\nabla}\right)=8\pi^2{\rm p}_1(E)\ ,\label{eq:pontry0}
\end{eqnarray}
where ${\rm p}_1(E)$ is the first Pontryagin number. Thus
$$
\frac{1}{8\pi^2}\int_\gamma \omega^{CS}\in\mathbb Z\ .
$$

\end{proof}
However when $E$ has rank 3 and is equipped with a quadratic form of signature $(2,1)$, we have a stronger result:
\begin{proposition}\label{pro:pontry2}
	If $E$ is equipped with a quadratic  $q$ form of signature $(2,1)$, if $\nabla_1$ and $\nabla_2$ are two $q$-connections gauge equivalent (when the gauge group is $\mathsf{SO}_0(2,1)$), then 
	\begin{eqnarray}
\int_\gamma \omega^{CS}=0\ .\label{eq:pontry2}
\end{eqnarray}
\end{proposition}
\begin{proof}
	We can always choose a Euclidean metric on $E$ and combining with the quadratic form $q$, we obtain, by finding a common orthogonal basis, a splitting of $E$ as two subbundles
	$$
	E=L\oplus L^\perp\ ,
	$$
	where $L$ is the line bundle generated by the timelike vector of the base. Denoting by ${\rm p}_1$ the first Pontryagin class on $M\times S^1$, by additivity of the first Pontryagin class \cite{Milnor:1974aa} we get
	$$
	{\rm p}_1(E)={\rm p}_1(L) +{\rm p}_1(L^\perp)\ .
	$$
	Observe now that since $L$ and $L^\perp$ have dimension 1 and 2 respectively, ${\rm p}_1(L)={\rm p}_1(L^\perp)=0$. Thus ${\rm p}_1(E)=0$. The result now follows.
\end{proof}

From  equation \eqref{eq:pontry}, one deduces that $\omega^{CS}$ is exact:  the integral of $\omega^{CS}$ on a path of connections only depends on the end points of the path. We then define, given two connections $\nabla_1$ and $\nabla_2$, the {\em Chern--Simons invariant} of the pair $(\nabla_1,\nabla_2)$ by 
$$
\CS{\nabla_1,\nabla_2}\defeq\frac{1}{8\pi^2}\int_\gamma \omega^{CS}\in\mathbb Z\ ,
$$
Moreover, if $\nabla_1$ and $\nabla_2$ are two  $q$-flat connections with holonomy $\rho_1$ and $\rho_2$ with values in $\mathsf{SO}_0(2,1)$, we define
$$
\CS{\rho_1,\rho_2}\defeq\CS{\nabla_1,\nabla_2}\ .
$$
The definition in unambiguous: observe that by proposition \ref{pro:pontry2}, $\CS{\rho_1,\rho_2}$ is well defined for representations defined up to conjugacy by the group $\mathsf{SO}_0(2,1)$.
\subsection{Some properties}
As an immediate consequence of the definition, we have the following  easy results:
\begin{lemma}{\sc[Chasles relation]}\label{lem:Chas}
We have
\begin{eqnarray*}
\CS{\nabla_1,\nabla_2}=\CS{\nabla_1,\nabla_3}+\CS{\nabla_3,\nabla_2}\ .	
\end{eqnarray*}
\end{lemma}

\begin{lemma}\label{lem:flatconn}
	Let $\gamma$ be a path of flat connections joining $\nabla_1
$ to $\nabla_2$, then
	$$
	\CS{\nabla_1,\nabla_2}=\int_\gamma \omega^{CS}=0\ .
	$$
\end{lemma}

\begin{lemma}\label{lem:deg} Let $M$ and $N$ be two closed 3-manifolds.
	Let $\pi$ be a degree $d$ map from $M$ to $N$, $\nabla_1$ and $\nabla_2$ two connections on a bundle over $N$,  
	$$
	\CS{\pi^*(\nabla_1),\pi^*(\nabla_2)}= d\ \CSN{\nabla_1,\nabla_2}\ .
	$$
\end{lemma}

\section{Volumes and Chern--Simons}

\subsection{Volumes}
Let us consider the bundle $E\defeq\mathsf T M_0$ over $M_0$, where $M_0$ is the universal cover of $\PSL$. The bundle $E$ carries a natural Lie bracket defined fiberwise.  Let us now consider the 3-form on $M_0$ defined by 
$$
\Omega(X,Y,Z)=\operatorname{trace}(X[Y,Z)])\ ,
$$
where the trace is taken on the (adjoint) 3-dimensional representation of $\PSL$. Now observe that $\Omega$ is invariant by the isometry group. An explicit computation relates $\Omega$ to the volume form:
\begin{eqnarray}
\Omega(X,Y,Z)=\frac{1}{2}\operatorname{Vol}(X,Y,Z)\ .	
\end{eqnarray}

\subsection{Back to Chern-Simons}\label{sec:rep} The bundle $E=\mathsf T M_0$ can be trivialized in two ways: the {\em left trivialization} for which left invariant vector fields are constant,  the {\em right  trivialization} for which right invariant vector fields are constant.  In other words,  $E$ carries two flat connections $\nabla_L$ and $\nabla_R$ which are respectively the left and right invariant connections. These two connections exist on the tangent bundle of any AdS manifold  $M$ and the holonomy of  $\nabla_L$ (on $M$) is $\rho$, while the holonomy of $\nabla_R$ is $\sigma$. 

Our first goal is to show the following proposition that initiated \cite[Theorem 2.10]{Tholozan:2023}.
\begin{proposition}\label{pro:vol}
On a closed AdS manifold $M$ :
	$$\CS{\nabla_L,\nabla_R}=-\frac{1}{24\pi^2}\operatorname{Vol}(M)\ .$$
\end{proposition}
One should not take the sign too seriously here, it is a matter of convention.
\begin{proof}
We first have
$$
\nabla_L-\nabla_R=A\ , 
$$
where $A$, related to the Maurer--Cartan form, is the element of $\Omega^1(X,\operatorname{End}(E))$ given by
$$
A(X): Y\mapsto [X,Y]\ .
$$
When we compute the curvature we get
$$
R^{\nabla_L}=R^{\nabla_R}+ {\rm d}^{\nabla_R} A + \frac{1}{2}[A\wedge A]\ ,
$$ where 
$$
[A\wedge B](X,Y)=[A(X),B(Y)]-[A(Y),B(X)]\ .
$$
and since both $\nabla_L$ and $\nabla_L$ are trivial hence flat, we obtain the {\em Maurer--Cartan equation}
$$
{\rm d}^{\nabla_R} A + \frac{1}{2}[A\wedge A]=0\ .
$$
Let us now consider the affine path between $\nabla_R$ and $\nabla_L$ given by 
$$
\gamma:t\mapsto \nabla^t=t\nabla_L + (1-t)\nabla_R=\nabla_R +t A\ .
$$
Then one then sees that 
$$
R^{\nabla_t}= t\ {\rm d}^{\nabla_R}A + t^2\frac{1}{2}[\ A\wedge A]= \frac{(t^2-t)}{2}[A\wedge A]\ . 
$$
It follows that
\begin{eqnarray*}
	\int_{\gamma}\omega^{CS}=-\frac{1}{3}\int_M \operatorname{trace}(A [A\wedge A])=-\frac{1}{3}\operatorname{Vol}(M)\ .
\end{eqnarray*}
The result now follows.  \end{proof}
\section{Computing Chern-Simons invariants} We finally need to compute explicitly the Chern--Simons invariant using Guritaud--Kassel description.

Let us consider the circle bundle $M(S,k)$ of Euler characteristic $k$ over a surface of Euler characteristic $e$, and let $\rho$ be a representation of $\pi_1(S)$ in $\PSL$ of Euler class $f$.

Our goal is to show
\begin{proposition}\label{pro:Mgk} We have for $M=M(S,k)$
	$$
	\CS{\rho,\id}= -\frac{f^2}{6k}\  \ .
	$$
\end{proposition}

We start with a special case. Let us first remark that if $S_e$ is a surface of Euler characteristic $e$, then the unit tangent bundle ${\mathsf U}S_e$ of $S_e$ identifies with $M(S_e,e)$.
Moreover, if $S_e$ is equipped with a hyperbolic metric, and if we denote by $\rho_e$ the associated monodromy,  then 
$M(S_e,e)$ is modelled on $\PSL$ and the pair $(\rho,\sigma)$ in the Kulkarni--Raymond description is $(\rho_e,\id)$ where $\id$ is the trivial representation. 

 We therefore obtain using  this observation

\begin{lemma}\label{pro:hyp}
	We have on M=$M(S_e,1)$
	$$
	\CS{\rho_e,\id}= -\frac{1}{6}e^2 \ . 
	$$
\end{lemma}

\begin{proof}
We first have 
	$$\operatorname{Vol}{\mathsf U}S_e= 4\pi^2 e\ .
	$$
Thus on $N=M(S_e,e)$ by proposition \ref{pro:vol}, we have 
$$
	\CSN{\rho_e,\id}= -\frac{1}{6}e  \ . 
	$$
Since we have a degree $e$ map from $M=M(S_e,1)$ to  $N=M(S_e,e)$ we get the result from lemma \ref{lem:deg}.
\end{proof}
Finally

\begin{proof}[Proof of proposition \ref {pro:Mgk}] By lemma \ref{lem:flatconn}   and Goldman's theorem \cite{Goldman:1980} that describes connected components of the space of representations of $\pi_1(S)$ in $\PSL$, $\CS{\rho, \id}$ only depends on  $e$.

Let us just choose a degree 1 map $\pi$ from the surface $S$ of genus $g$, to the surface $S_f$ of Euler characteristic $f$. Let us equipped $S_e$ with a hyperbolic metric with monodromy $\rho_e$. 

We then deduce a degree 1 map from $\pi$ from $M(S,1)$ to $M(S_f,1)$ and $\pi_*\rho_f$ has Euler class $f$. Thus using lemma \ref{lem:deg}, we have on $M=M(S,1)$
$$
\CS{\rho,\id}=\CS{\rho_f,\id}= -\frac{1}{6}f^2  \\ ,
$$
where the last inequality comes from proposition \ref{pro:hyp}. Finally since we have a degree $k$ map from $M(S,1)$ to $M(S,k)$ we get the result from lemma \ref{lem:deg}. \end{proof}
\subsection{Tholozan's volume formula}
We can now prove formula \eqref{eq:Tholoz}. Let $S$ be a surface of negative characteristic $e$. Let $M=M(S,k)$ be equipped with an AdS structure described by  the representations $\rho$ of Euler characterstic $e$ and $\sigma$ of Euler characteristic $f$.
It follows from proposition \ref{pro:vol}, that
$$
\CS{\rho,\sigma}= -\frac{1}{24\pi^2}\operatorname{Vol}(M)\  \ .
$$
On the other hand, from the definition of the Chern--Simons invariant, 
we have 
$$
\CS{\rho,\sigma}=\CS{\rho,\id}-\CS{\sigma,\id}=\frac{1}{6k}\left(f^2-e^2\right)\ .
$$
where we used Chasles relation (lemma \ref{lem:Chas}) in the first equality and the last inequality comes from proposition \ref{pro:Mgk}. It follows  that
$$
\operatorname{Vol}(M)=\frac{4\pi^2}{k}\left(e^2-f^2\right)\  \ . 
$$
which is what we wanted to prove.
\bibliographystyle{amsplain}

\begin{thebibliography}{10}

\bibitem{Andersson:2007tt}
Lars Andersson, Thierry Barbot, Riccardo Benedetti, Francesco Bonsante, William~M Goldman, Fran{\c c}ois Labourie, Kevin~P Scannell, and Jean-Marc Schlenker, \emph{{Notes on a paper of Mess}}, Geometriae Dedicata \textbf{126} (2007), no.~1, 47--70.

\bibitem{Chern:1974aa}
Shiing~Shen Chern and James Simons, \emph{Characteristic forms and geometric invariants}, Ann. of Math. (2) \textbf{99} (1974), 48--69. \MR{353327}

\bibitem{Goldman:1980}
William~M Goldman, \emph{{Discontinuous group and the Euler Class.}}, Ph.D. thesis, University of Berkeley, California, 1980.

\bibitem{Gueritaud:2017aa}
Fran{\c{c}}ois Gu{\'{e}}ritaud and Fanny Kassel, \emph{Maximally stretched laminations on geometrically finite hyperbolic manifolds}, Geom. Topol. \textbf{21} (2017), 693--840.

\bibitem{Kassel:2009aa}
Fanny Kassel, \emph{Quotients compacts d'espaces homog{\`e}nes r{\'e}els ou $p$-adiques}, Ph.D. thesis, Universit{\'e} Paris-Sud, 2009.

\bibitem{Klingler:1996aa}
Bruno Klingler, \emph{Compl\'{e}tude des vari\'{e}t\'{e}s lorentziennes \`a courbure constante}, Math. Ann. \textbf{306} (1996), no.~2, 353--370. \MR{1411352}

\bibitem{Kulkarni:1985aa}
Ravi~S. Kulkarni and Frank Raymond, \emph{{$3$}-dimensional {L}orentz space-forms and {S}eifert fiber spaces}, J. Differential Geom. \textbf{21} (1985), no.~2, 231--268. \MR{816671}

\bibitem{Labourie:2013ka}
Fran{\c c}ois Labourie, \emph{{Lectures on representations of surface groups}}, Zurich Lectures in Advanced Mathematics, European Mathematical Society (EMS), Z\"urich, 2013.

\bibitem{Mess:2007kx}
Geoffrey Mess, \emph{{Lorentz spacetimes of constant curvature}}, Geometriae Dedicata \textbf{126} (2007), 3--45.

\bibitem{Milnor:1974aa}
John~W. Milnor and James~D. Stasheff, \emph{Characteristic classes}, Annals of Mathematics Studies, vol. No. 76, Princeton University Press, Princeton, NJ; University of Tokyo Press, Tokyo, 1974. \MR{440554}

\bibitem{Salein:1999aa}
Francois Salein, \emph{Vari{\'e}t{\'e}s anti-de sitter de dimension 3}, Ph.D. thesis, ENS Lyon, 1999.

\bibitem{Salein:2000vm}
\bysame, \emph{{Vari{\'e}t{\'e}s anti-de Sitter de dimension 3 exotiques}}, Annales de l'Institut Fourier \textbf{50} (2000), no.~1, 257--284.

\bibitem{Tholozan:2023}
Nicolas Tholozan, \emph{Chern-simons theory and cohomological invariants of representation varieties}, arXiv:2309.09576.

\bibitem{Tholozan:2014aa}
\bysame, \emph{Uniformisation des vari{\'e}t{\'e}s pseudo-riemanniennes localement homog{\`e}nes.}, Ph.D. thesis, Universit{\'e} de Nice, 2014.

\bibitem{Tholozan:2018aa}
\bysame, \emph{The volume of complete anti--de {S}itter 3-manifolds}, J. Lie Theory \textbf{28} (2018), no.~3, 619--642. \MR{3750161}

\end{thebibliography}
\providecommand{\bysame}{\leavevmode\hbox to3em{\hrulefill}\thinspace}
\providecommand{\MR}{\relax\ifhmode\unskip\space\fi MR }
\providecommand{\MRhref}[2]{%
  \href{http://www.ams.org/mathscinet-getitem?mr=#1}{#2}
}
\providecommand{\href}[2]{#2}
   \def\MR#1{}

\end{document}